\documentclass[graybox]{svmult}

\usepackage{helvet}         
\usepackage{courier}        
\usepackage{type1cm}        
\usepackage{makeidx}         
\usepackage{graphicx}        
\usepackage{multicol}        
\usepackage[bottom]{footmisc}
\usepackage{amsfonts}
\usepackage{amsmath}
\usepackage{graphics}
\usepackage{stmaryrd}
\usepackage{amssymb}
\usepackage{mathrsfs}
\usepackage{tikz}
\DeclareMathOperator{\tr}{tr}
\newcommand{\paper}[6]{
#1, #2. #3 \textbf{#4}, #5 (#6)}
\newcommand{\book}[7]{
#1, #2. #3, vol. #4 (#5, #6, #7)}
\newcommand{\proc}[9]{
#1, #2, in #3, ed. by #4. #5, vol. #6 (#7, #8, #9)}
\newcommand{\procc}[8]{
#1, #2, in #3, #4, vol. #5 (#6, #7, #8)}
\newcommand{\proccc}[6]{
#1, #2, in #3, ed. by #4. #5 (#6)}

\begin{document}

\title*{Coherent state representations of the holomorphic automorphism group of a quasi-symmetric Siegel domain of type III}\titlerunning{Coherent state representations over a quasi-symmetric Siegel domain}
\author{Koichi Arashi}
\institute{Koichi Arashi \at Department of Mathematics, Tokyo Gakugei University, Nukuikita 4-1-1, Koganei,Tokyo 184-8501, Japan, \email{arashi@u-gakugei.ac.jp}
}
\maketitle

\abstract{
The aim of this paper is to study the classification of generic coherent state representations of a Lie group and its connection to the structure theory of K\"{a}hler algebras.
The group we deal with is the holomorphic automorphism group of a quasi-symmetric Siegel domain of type III.
A method using the generalized Iwasawa decomposition is provided.}

\section{Introduction}
A highest weight representation is a ubiquitous concept in the representation theory.
We recall a simple complex geometric condition in \cite{Lisiecki90}, which characterizes unitary highest weight representations of semisimple Lie groups and makes sense in the case of arbitrary Lie groups.
For a unitary representation $(\pi,\mathcal{H})$ of a connected Lie group $G$, we can regard the projective space $\mathcal{P}(\mathcal{H})$ as a (possibly infinite-dimensional) K\"{a}hler manifold.
Then Lisiecki called a $G$-orbit in $\mathcal{P}(\mathcal{H})$ a coherent state orbit (CS orbit for short) if it is a complex submanifold of $\mathcal{P}(\mathcal{H})$ and $\pi$ a coherent state representation (CS representation for short) if it admits a CS orbit in $\mathcal{P}(\mathcal{H})$ which does not reduce to a point (see \cite{Lisiecki95}).
The classification of generic CS representations is established for unimodular Lie groups \cite{Lisiecki91}.
By Neeb\cite{Neeb}, a generalization of unitary highest weight representations and its relation to the CS representations are studied in the setting of Lie groups which have compactly embedded Cartan subalgebras.

As an irreducible unitary highest weight representation of a non-compact Lie group can be realized on the space of vector-valued holomorphic functions on a bounded symmetric domain, 
for a complex bounded domain $\mathcal{D}$ and Lie group $G$ acting transitively on $\mathcal{D}$, an irreducible CS representation of a Lie group $G$ realized on the space of such functions can be considered as a straightforward generalization of a unitary highest weight representation or its restriction.
Let $\mathrm{Aut}_{hol}(\mathcal{D})$ be the holomorphic automorphism group, $B$ a maximal connected triangular Lie subgroup of the identity component $\mathrm{Aut}_{hol}(\mathcal{D})^o$.
For the scalar valued case, Ishi \cite{Ishi11} classified unitarizations of holomorphic multiplier representations of $B$.
For a real algebraic group $G\subset\mathrm{Aut}_{hol}(\mathcal{D})$, which contains $B$, the author proved in \cite{Arashi20} that if all holomorphic multipliers are given, then the classification of unitarizations for the identity component $G^o$ is reduced to the one for $B$ and taking look at certain one-dimensional representation of a maximal compact subgroup $U$ of $G^o$.
Note that $G^o$ is topologically decomposed into the direct product of $B$ and $U$ (see \cite[Chapter 4, Theorem 4.7]{Enc}), which is called a generalized Iwasawa decomposition.
In the previous work \cite{Arashi}, the author classified irreducible CS representations of the identity component $G^o$ of a certain twelve-dimensional real algebraic group $G$.
For a closely related work for split solvable Lie groups, we refer to \cite{Ishi05}.

In this paper, we deal with generic CS representations of the universal covering group $G$ of $\mathrm{Aut}_{hol}(\mathcal{D})^o$ when $\mathcal{D}$ is a quasi-symmetric Siegel domain of type III$_{\nu,r}$ with $\nu\geq 3, r\geq 2$.

In Section \ref{sec:Siegel}, we determine the structure of $\mathrm{Lie}(G)$ by using basic results of the structure theory in \cite{Satake}.
In Section \ref{sec:Realization}, we review briefly some facts about a realization of an irreducible CS representation studied in \cite{Lisiecki90, Lisiecki91}.
In Section \ref{sec:Algebraic}, we determine a basic algebraic structure of CS orbit by using the structure theory of K\"{a}hler algebras studied in \cite{Borel,DN,GPV,Matsushima,Shima,VGP}.
In Section \ref{sec:Holomorphic}, we determine $G$-equivariant holomorphic vector bundles of which a generic CS representation is realized in the space of holomorphic sections by using the holomorphic embeddings of K\"{ahler} homogeneous manifolds in \cite{DN}, the classification of the structures of $G$-equivariant holomorphic vector bundles in \cite{TW}, and the rigidity of invariant complex structures of a K\"{a}hler solvmanifold studied in \cite{Dotti}.
In Section \ref{sec:Equivalence}, we see a relationship between a generic CS representation of $G$ and the one of a maximal connected triangular Lie subgroup $B$ of $G$ together with a finite-dimensional representation of a maximal compact subgroup of $G$.

\section{Siegel domain and holomorphic complete vector fields}
\label{sec:Siegel}
Let $\nu\geq3,r\geq 2$ be positive integers, $\mathop{\mathrm{Sym}}_\nu(\mathbb{R})$ ($\mathop{\mathrm{Sym}}_\nu(\mathbb{C})$) the space of $\nu$-by-$\nu$ real (complex) symmetric matrices, $\Omega\subset \mathop{\mathrm{Sym}}_\nu(\mathbb{R})$ the set of all positive-definite matrices, and $Q:\mathop{\mathrm{Mat}}_{\nu,r}(\mathbb{C})\times\mathop{\mathrm{Mat}}_{\nu,r}(\mathbb{C})\rightarrow\mathop{\mathrm{Sym}}_\nu(\mathbb{C})$ a Hermitian map defined by $Q(U,V):=(U\overline{{}^tV}+\overline{V}{}^t U)/4$, where for a complex matrix $X$, we denote by $\overline{X}$ and ${}^tX$ the complex conjugate and the transpose of $X$, respectively.
We consider the following quasi-symmetric Siegel domain $\mathcal{D}$ of type III$_{\nu,r}$:
\begin{equation*}
\mathcal{D}:=\left\{(Z,V)\in\mathrm{Sym}_\nu(\mathbb{C})\times\mathrm{Mat}_{\nu,r}(\mathbb{C})\mid \mathop{\mathrm{Im}}Z-Q(V,V)\in\Omega\right\}.
\end{equation*}
Let $\mathrm{Aut}_{hol}(\mathcal{D})$ be the holomorphic automorphism group of $\mathcal{D}$.
The group admits a unique structure of a Lie group which induces a $C^\infty$-group action on $\mathcal{D}$ and is compatible with the compact-open topology.
By the structure theory of the Lie algebra of the holomorphic automorphism group of a Siegel domain, we can determine $\mathfrak{g}:=\mathop{\mathrm{Lie}}(\mathrm{Aut}_{hol}(\mathcal{D}))$.
We refer the reader to {\cite[Chapter V]{Satake}} for preliminary results.
For $T\in\mathfrak{gl}(\nu,\mathbb{R})$, $K\in\mathop{\mathfrak{u}}(r)$, $V\in\mathrm{Mat}_{\nu,r}(\mathbb{C})$, and $U\in\mathrm{Sym}_\nu(\mathbb{R})$, put
\begin{equation*}
\iota(T,K,V,U)=\left(\begin{array}{cccc}\mathop{\mathrm{Re}}K&0&-\mathop{\mathrm{Im}}K&\mathop{\mathrm{Re}}{}^tV\\
\mathop{\mathrm{Im}}V&T&\mathop{\mathrm{Re}}V&U\\
\mathop{\mathrm{Im}}K&0&\mathop{\mathrm{Re}}K&-\mathop{\mathrm{Im}}{}^tV\\
0&0&0&-{}^tT\end{array}\right)\in\mathrm{Mat}_{2(\nu+r)}(\mathbb{R}).
\end{equation*}
\begin{proposition}
The Lie algebra $\mathfrak{g}$ is isomorphic to the linear Lie algebra $\widetilde{\mathfrak{g}}$ consists of all matrices $\iota(T,K,V,U)$.
\end{proposition}
\begin{proof}
We shall give a sketch of the proof.
In the following remark, we will mention an explicit correspondence between the Lie algebras.
It is known that $\mathfrak{g}$ can be identified with the Lie algebra $\mathfrak{X}$ of all complete holomorphic vector fields over $\mathcal{D}$.
Let $\partial\in\mathfrak{X}$ be given by the one parameter subgroup $\mathcal{D}\ni(Z,V)\mapsto(e^{-t}Z,e^{-t/2}V)\in\mathcal{D}\,(t\in\mathbb{R})$ of $\mathrm{Aut}_{hol}(\mathcal{D})$.
For $\lambda\in\mathbb{R}$, put $\mathfrak{X}_\lambda:=\{X\in\mathfrak{X}\mid[\partial,X]=\lambda X\}$.
It is known \cite[Theorems 4 and 5]{KMO} that the Lie algebra $\mathfrak{X}$ has a grading structure
\begin{equation*}
\mathfrak{X}=\mathfrak{X}_{-1}+\mathfrak{X}_{-1/2}+\mathfrak{X}_0+\mathfrak{X}_{1/2}+\mathfrak{X}_1
\end{equation*}
and in our case, $\mathfrak{X}_{-1}=\mathfrak{X}_{-1/2}=0$ since $\mathcal{D}$ is an irreducible quasi-symmetric Siegle domain and not symmetric.
By the explicit descriptions of $\mathfrak{X}_\lambda\,(\lambda=0,1/2,1)$ and the bracket relations among them, we can confirm the assertion.
\qed
\end{proof}
\begin{remark}
We can realize $\mathcal{D}$ in $\mathrm{Sym}_{\nu+r}(\mathbb{C})$ (see \cite{Ishi}), and the group $\exp(\widetilde{\mathfrak{g}})\subset GL(2(\nu+r),\mathbb{R})$ acts on $\mathcal{D}$ by generalized fractional linear transformations and is isomorphic to the identity component $\mathrm{Aut}_{hol}(\mathcal{D})^o$ of $\mathrm{Aut}_{hol}(\mathcal{D})$.
\end{remark}
\section{Realization of a CS representation}
\label{sec:Realization}
In this section, we refer the reader to \cite{Lisiecki90} for details about general theory of CS representations.
Let $G$ be a connected and simply connected Lie group with $\mathrm{Lie}(G)=\mathfrak{g}$.
Let $(\pi,\mathcal{H})$ be a generic (i.e., irreducible with discrete kernel) CS representation of $G$.
Then it is known that $G$ acts on a CS orbit almost effectively.
For a holomorphic vector bundle $E$ over a complex manifold $M_0$, let us denote by $\Gamma^{hol}(M_0,E)$ the space of holomorphic sections.
Then the conjugate representation $\overline{\pi}$ can be realized in $\Gamma^{hol}(M,L)$ for some $G$-equivariant holomorphic line bundle $L$ over a CS orbit $M$, which admits a K\"{a}hler structure inherited from the Fubini-Study metric on the projective space $\mathcal{P}(\mathcal{H})$.
We get a connected subgroup $K$ of $G$ and an effective K\"{a}hler algebra $(\mathfrak{g},\mathfrak{k},\rho,j)$ with $G/K\simeq M$ and $\mathrm{Lie}(K)=\mathfrak{k}$ (see {\cite{Lisiecki95},\cite[Theorem 2.17]{RV}}) .
\section{Algebraic structures of a CS orbit}
\label{sec:Algebraic}
In this section, we refer the reader to \cite{GPV} and \cite{VGP} for basic definitions and preliminary results on K\"{a}hler algebra.
By Dorfmeister-Nakajima \cite[\S1, 1.4., LEMMA, \S2, 2.5., THEOREM]{DN}, there exists an abelian K\"{a}hler ideal $\mathfrak{a}$ and a K\"{a}hler subalgebra $\mathfrak{h}$ of $\mathfrak{g}$ such that $\mathfrak{g}=\mathfrak{a}\oplus\mathfrak{h},\,\rho(\mathfrak{a},\mathfrak{h})=0$,
and $(\mathfrak{h},\mathfrak{k},j)$ is a $j$-algebra.
Moreover, let $\mathfrak{k}_0$ be the maximal ideal of $\mathfrak{h}$ contained in $\mathfrak{k}$.
Then there exists an effective K\"{a}hler ideal $\mathfrak{h}'$ of $\mathfrak{h}$ such that $\mathfrak{h}=\mathfrak{h}'\oplus\mathfrak{k}_0$.
Let $\mathfrak{r}$ be a maximal abelian ideal of the first kind in $\mathfrak{h}'$, i.e., there exists $r\in\mathfrak{r}$ such that $[jx,r]=x\,(x\in\mathfrak{r})$.
Then by \cite[\S2]{VGP}, the real part of $\mathrm{ad}(jr)|_{\mathfrak{h}'}$ is semisimple and the eigenspace decomposition is given by $\mathfrak{h}'=\mathfrak{h}'_0+\mathfrak{h}'_{1/2}+\mathfrak{h}'_{1}$ with $\mathfrak{h}'_1=\mathfrak{r}$,
where for $\lambda\in\mathbb{R}$, we denote by $\mathfrak{h}'_\lambda$ the associated eigenspace.
Moreover, if we put $\mathfrak{s}:=\{x\in\mathfrak{h}'_0\mid[x,r]=0\}$, then $\mathfrak{s}$ is a reductive $j$-subalgebra and contains $\mathfrak{h}'\cap\mathfrak{k}$ and we have $\mathfrak{h}'_0=\mathfrak{s}\oplus j\mathfrak{r}$.

Let $\mathfrak{u}$ be the maximal compact subalgebra of $\mathfrak{s}$ which contains $\mathfrak{h}'\cap\mathfrak{k}$.
By Matsushima \cite{Matsushima}, the center $\mathfrak{z}(\mathfrak{s})$ is a $j$-subalgebra of $\mathfrak{s}$ and thus $\mathfrak{z}(\mathfrak{s})\subset\mathfrak{h}'\cap\mathfrak{k}$.
Hence by Borel \cite{Borel}, $(\mathfrak{u},\mathfrak{h}'\cap\mathfrak{k},j)$ is a $j$-subalgebra and $\mathfrak{h}'\cap\mathfrak{k}=\mathfrak{z}_\mathfrak{u}(\mathfrak{c})$ with $\mathfrak{c}\subset\mathfrak{u}$ abelian.
For a Lie algebra $\mathfrak{b}$, let us denote by $\mathrm{Int}(\mathfrak{b})$ the adjoint group of $\mathfrak{b}$.
Now $(\mathfrak{h}',\mathfrak{h}'\cap\mathfrak{k},j)$ is algebraic, and hence by Shima \cite{Shima}, $\exp\mathrm{ad}(\mathfrak{u})\subset \mathrm{Int}(\mathfrak{h}')$ is a maximal compact subgroup and there exists a triangular subalgebra $\mathfrak{t}$ of $\mathfrak{h}'$ satisfying the following conditions:
\begin{enumerate}
\item
$\mathfrak{h}'=\mathfrak{u}\oplus\mathfrak{t}$;
\item
there exists an abelian subalgebra $\mathfrak{a}_\mathfrak{t}$ of $\mathfrak{t}$ such that $\mathfrak{t}=[\mathfrak{t},\mathfrak{t}]\oplus\mathfrak{a}_\mathfrak{t}$ and $\mathrm{ad}(x):\mathfrak{h}'\rightarrow\mathfrak{h}'$ is real semisimple for all $x\in\mathfrak{a}_\mathfrak{t}$.
\end{enumerate}
We have the following lemma.
\begin{lemma}
$\exp\mathrm{ad}(\mathfrak{u}\oplus\mathfrak{k}_0)$ is a maximal compact subgroup of $\mathrm{Int}(\mathfrak{g})$.
\end{lemma}
\begin{proof}
The adjoint representation of $\mathfrak{g}$ on $(\mathfrak{a},\rho,j)$ is a normal symplectic representation when restricted to a subalgebra $\langle jx,x\rangle$ for any $x\in\mathfrak{a}_\mathfrak{t}$ (see \cite[Part III]{GPV}).
This implies that $\mathrm{ad}(jx):\mathfrak{g}\rightarrow\mathfrak{g}$ is real semisimple for all $x\in\mathfrak{a}_\mathfrak{t}$.
Hence $\mathrm{Int}(\mathfrak{g})$ is a product of a compact subgroup $\exp\mathrm{ad}(\mathfrak{u}\oplus\mathfrak{k}_0)$ and a connected triangular subgroup $\exp \mathrm{ad}(\mathfrak{t}\oplus\mathfrak{a})$, which implies that $\exp\mathrm{ad}(\mathfrak{u}\oplus\mathfrak{k}_0)$ is a maximal compact subgroup of $\mathrm{Int}(\mathfrak{g})$ (see \cite[Chapter 4, Propositions 4.4 and 4.5]{Enc}).
\qed
\end{proof}
Let $\mathfrak{g}_\lambda:=\{x\in\mathfrak{g}\mid[jr,x]=\lambda x\}\,(\lambda\in\mathbb{R})$.
It is known that all possible eigenvalues of $\mathrm{ad}(jr)$ are $0,\pm 1/2,1$ and $\dim\mathfrak{a}\cap\mathfrak{g}_{-1/2}=\dim\mathfrak{a}\cap\mathfrak{g}_{1/2}$.
\begin{proposition}
We have $\mathfrak{a}=\mathfrak{k}_0=0$.
\end{proposition}
\begin{proof}
Let $\mathfrak{u}'\subset\mathfrak{g}$ be a subalgebra such that $\exp\mathrm{ad}(\mathfrak{u}')\subset\mathrm{Int}(\mathfrak{g})$ is a maximal compact subgroup.
Then there exists $g\in\mathrm{Int}(\mathfrak{g})$ such that $g\mathop{\mathrm{ad}}(\mathfrak{u}')g^{-1}=\mathrm{ad}(\mathfrak{u}\oplus\mathfrak{k}_0)$.
Since $\mathfrak{z}(\mathfrak{g})=0$, we have $g\mathfrak{u}'=\mathfrak{u}\oplus\mathfrak{k}_0$ and hence $[g^{-1}jr,\mathfrak{u}']=0$.
Letting $\mathfrak{u}'\subset\mathfrak{g}$ be defined by $\{\iota(T,K,0,0)\mid T\in\mathfrak{o}(\nu), K\in\mathfrak{u}(r)\}$, we see that $g^{-1}jr$ is defined by $\iota(bI_\nu,iaI_\nu,0,cI_\nu)\in\widetilde{\mathfrak{g}}$ for some $a,b,c\in\mathbb{R}$.
We see that $\mathrm{ad}(g^{-1}jr)$ has eigenvalues $b\pm ia, 2b$ and it follows that $a=0$ and $b=1/2$.
Therefore for some $g\in \mathrm{Int}(\mathfrak{g})$, the element $g^{-1}jr$ corresponds to $\iota(1/2I_\nu,0,0,0)$.

Let us consider an effective $j$-algebra $(\mathfrak{g},\mathfrak{k}_\mathcal{D},j_\mathcal{D})$ associated with $\mathcal{D}$ equipped with the Bergman metric, and let $\mathfrak{r}_\mathcal{D}:=\{\iota(0,0,0,U)\mid U\in \mathrm{Sym}_\nu(\mathbb{R})\}$.
Then we have $\dim\mathfrak{r}=\dim\mathfrak{r}_\mathcal{D}$, $\dim\mathfrak{u}\oplus\mathfrak{k}_0=\dim\mathfrak{k}_\mathcal{D}$, $\dim(\mathfrak{a}\cap\mathfrak{g}_0)\oplus j\mathfrak{r}\oplus\mathfrak{s}\oplus\mathfrak{k}_0=\dim j\mathfrak{r}_\mathcal{D}\oplus\mathfrak{k}_\mathcal{D}$, from which we see that $\mathfrak{a}\cap\mathfrak{g}_0=0$ and $\mathfrak{s}=\mathfrak{u}$.
Also $\mathfrak{g}_{-1/2}=0$ implies that $\mathfrak{a}=0$.
Now for $\mathfrak{u}'=\mathfrak{k}_\mathcal{D}$, the equality $g\mathfrak{u}'=\mathfrak{u}\oplus\mathfrak{k}_0$ implies that $g^{-1}\mathfrak{k}_0$ is contained in the kernel of ineffectiveness, which shows that $\mathfrak{k}_0=0$.
\qed
\end{proof}

It is known \cite[\S 4]{VGP} that there exist $r_1,r_2,\cdots,r_\nu\in\mathfrak{r}$ and subspaces $\mathfrak{r}_{\alpha\beta}\subset\mathfrak{r}\,(1\leq \alpha\leq\beta\leq\nu)$ such that
\begin{equation*}
\mathfrak{r}=\bigoplus_{\alpha\leq \beta}\mathfrak{r}_{\alpha\beta},\quad \mathfrak{r}_{\alpha\alpha}=\mathbb{R}r_\alpha,\quad r=\sum_{\alpha=1}^\nu r_\alpha
\end{equation*}
and for any $x\in\mathfrak{r}_{\beta\gamma}$, we have $[jr_\alpha,x]=(\delta_{\alpha\beta}+\delta_{\alpha\gamma})x/2$, $[jr_\alpha,jx]=(\delta_{\alpha\beta}-\delta_{\alpha\gamma})jx/2$.
Moreover, after an essential change of $j$, we may assume that $\{\mathrm{ad}(jr_\alpha)\}_{1\leq \alpha\leq \nu}$ are commutative and real semisimple.
For $\Lambda=(\Lambda_1,\Lambda_2,\cdots,\Lambda_\nu)\in\mathbb{R}^\nu$, let $\mathfrak{g}^\Lambda:=\{x\in\mathfrak{g}\mid \mathrm{ad}(jr_\alpha)x=\Lambda_\alpha x\text{ for all }1\leq \alpha\leq \nu\}$, and for $1\leq \alpha\leq \nu$, let $\Lambda^{(\alpha)}\in\mathbb{R}^\nu$ be given by $\Lambda_\beta^{(\alpha)}=\delta_{\alpha\beta}$.
Put $\mathfrak{g}':=\mathfrak{r}+j\mathfrak{r}+\mathfrak{k}$.
It is also known that for any $\Lambda\in\mathbb{R}^\nu$, one has $j\mathfrak{g}^\Lambda\subset\mathfrak{g}^\Lambda+\mathfrak{g}'$.
The subalgebra $\mathfrak{u}$ contains ideals $\mathfrak{u}_1\simeq \mathfrak{o}(\nu)$ and $\mathfrak{u}_2\simeq\mathfrak{su}(r)$.
We have the following proposition.
\begin{proposition}
$\mathfrak{u_1}$ is contained in $\mathfrak{k}$.
\end{proposition}
\begin{proof}
Let $1\leq \alpha<\beta\leq \nu$.
Then there exists $x\in\mathfrak{g}$ such that $\mathfrak{g}^{(\Lambda^{(\beta)}-\Lambda^{(\alpha)})/2}=\mathbb{R}x$.
Also there exists $u\in\mathfrak{u}_1$, and $y\in\mathfrak{g}^{(\Lambda^{(\alpha)}-\Lambda^{(\beta)})/2}$ such that $x+u=y$.
We have $ju=-jx+jy=\lambda x+z+jy\in \mathfrak{u}_1$ for some $\lambda\in\mathbb{R}$ and $z\in\mathfrak{g}'$.
Since $\lambda x+z+jy, \lambda x-\lambda y\in\mathfrak{u}_1$, we have $z+jy+\lambda y\in\mathfrak{u}\cap \mathfrak{g}'=\mathfrak{k}$.
From the equality $ju=-\lambda u+(z+jy+\lambda y)$ we see that $\rho(ju,u)=0$, and hence $u\in\mathfrak{k}$.
Since $1\leq \alpha<\beta\leq \nu$ are arbitrary, we conclude that $\mathfrak{u}_1\subset\mathfrak{k}$.
\qed
\end{proof}
Now we have $\mathfrak{k}=\mathfrak{u}_1+\mathfrak{z}_{\mathfrak{u}_2}(\mathfrak{c}_2)+\mathfrak{z}(\mathfrak{u})$ with $\mathfrak{c}_2\subset\mathfrak{u}_2$ abelian.
\section{Holomorphic vector bundles}
\label{sec:Holomorphic}
For a real vector space or a real Lie algebra $V$, we denote by $V_\mathbb{C}$ the complexification $V\otimes_\mathbb{R}\mathbb{C}$.
Define
\begin{eqnarray*}
\quad\mathfrak{g}_-:=\{x+ijx\in\mathfrak{g}_\mathbb{C}\mid x\in\mathfrak{g}\}+\mathfrak{k}_\mathbb{C},\quad \mathfrak{g}_-^*:=\mathfrak{g}_-+\mathfrak{u}_\mathbb{C},\\
\mathfrak{u}_-:=\{x+ijx\in\mathfrak{u}_\mathbb{C}\mid x\in\mathfrak{u}\}+\mathfrak{k}_\mathbb{C}.
\end{eqnarray*}
Let $U:=\exp\mathfrak{u}$ be the corresponding subgroup of $G$.
The rigidity of invariant complex structures of a K\"{a}hler solvmanifold studied in \cite{Dotti} implies that there is a biholomorphism from $G/U$ into $\mathcal{D}$.
Therefore, we only need to consider the action of $G$ on $\mathcal{D}$ given by some automorphism $\varphi$ of $\mathrm{Aut}_{hol}(\mathcal{D})^o$ as follows: $g.z:=\varphi(p(g))z\,(g\in G, z\in\mathcal{D})$, where $p:G\rightarrow\mathrm{Aut}_{hol}(\mathcal{D})^o$ is the covering.
Let $\varphi_0$ be an automorphism of $\widetilde{\mathfrak{g}}$ given by $\varphi_0(\iota(T,K,V,U)):=\iota(T,\overline{K},\overline{V},-U)$.
\begin{proposition}
Any automorphism $\varphi$ of $\widetilde{\mathfrak{g}}$ is written by $\varphi=g\varphi_0^{\varepsilon}$ for some $g\in\mathrm{Int}(\widetilde{\mathfrak{g}})$ and $\varepsilon\in\{0,1\}$.
\end{proposition}
\begin{proof}
Let $\mathfrak{l}:=\{\iota(T,K,0,0)\mid T\in\mathfrak{gl}(\nu,\mathbb{R}), K\in\mathfrak{u}(r)\}$.
Then there exists $g\in\mathrm{Int}(\widetilde{g})$ such that $g\varphi(\mathfrak{l})\subset\mathfrak{l}$.
It is known that the orders of the outer automorphism groups of $\mathfrak{sl}(\nu,\mathbb{R})$ and $\mathfrak{su}(r)$ are $2$ (see \cite{Murakami}).
Therefore if we denote by $\sigma_1$ and $\sigma_2$ the automorphisms of $\mathfrak{sl}(\nu,\mathbb{R})$ and $\mathfrak{su}(r)$ given by the negative transposes, we can further assume that $g\varphi(\iota(T,K,0,0))=\iota(\sigma_1^{\varepsilon_1}(T),\sigma_2^{\varepsilon_2}(K),0,0)\,(T\in\mathfrak{sl}(\nu,\mathbb{R}),K\in\mathfrak{su}(r))$ with $\varepsilon_1,\varepsilon_2\in\{0,1\}$.
Moreover, we see that $\varphi(\iota(I_\nu,0,0,0))=\iota(I_\nu,0,0,0)$, $\varphi(\widetilde{\mathfrak{g}}_\lambda)=\widetilde{\mathfrak{g}}_\lambda\,(\lambda=0,1/2,1)$, where we put $\widetilde{\mathfrak{g}}_\lambda:=\{X\in\widetilde{\mathfrak{g}}\mid\mathrm{ad}(\iota(I_\nu,0,0,0))X=\lambda X\}$.
This leads to $g\varphi(\iota(T,K,V,U))=\iota(T,K,\lambda V,\lambda^2 U)$ or $\iota(T,\sigma_2(K),\lambda \overline{V},-\lambda^2 U)$ for some $\lambda\in\mathbb{R}$, which shows the assertion.
\end{proof}
The automorphism $\varphi_0$ lifts to an automorphism of $G$, which induces a biholomorphism from $\mathcal{D}$ to the conjugate manifold $\overline{\mathcal{D}}$.
Take a reference point $(iI_\nu,0)\in\mathcal{D}$.
Now we may assume that $\mathfrak{u}$ and $\mathfrak{g}_-^*$ are defined by $\widetilde{\mathfrak{u}}:=\{\iota(T,K,0,0)\mid T\in\mathfrak{o}(\nu),K\in \mathfrak{u}(r)\}$ and
\begin{equation*}
\widetilde{\mathfrak{g}_-^*}:=\left\{\begin{array}{c}\iota(T_1,K_1,V,T_2+{}^tT_2)\\+i\iota(T_2,K_2,i\overline{V},-T_1-{}^tT_1)\in\widetilde{\mathfrak{g}}_\mathbb{C}\end{array}\mid\begin{array}{c} T_1,T_2\in\mathfrak{gl}(\nu,\mathbb{R}), K_1, \\K_2\in\mathfrak{u}(r), V\in\mathrm{Mat}_{\nu,r}(\mathbb{C})\end{array}\right\}
\end{equation*}
or its complex conjugate, respectively.

Let $G_\mathbb{C}$ be a connected and simply connected Lie group with $\mathrm{Lie}(G_\mathbb{C})=\mathfrak{g}_\mathbb{C}$, and $G_-:=\exp\mathfrak{g}_-$, $G_-^*:=\exp\mathfrak{g}_-^*$, $U_-:=\exp\mathfrak{u}_-$, $U_\mathbb{C}:=\exp\mathfrak{u}_\mathbb{C}$ the corresponding subgroups of $G_\mathbb{C}$.
By \cite[\S 7]{DN}, the natural homomorphism from $G$ into $G_\mathbb{C}$ induces the holomorphic embeddings
\begin{equation*}
G/K\hookrightarrow G_\mathbb{C}/G_-,\quad G/U\hookrightarrow G_\mathbb{C}/G_-^*.
\end{equation*}
Any one-dimensional representation of $K$ is given by $(\chi,\mathbb{C}^\chi):=(1\boxtimes\chi_1\boxtimes\chi_2,\mathbb{C}\otimes\mathbb{C}^{\chi_1}\otimes\mathbb{C}^{\chi_2})$, where $(1,\mathbb{C})$ is the trivial representation of $\exp \mathfrak{u}_1$, and $(\chi_1,\mathbb{C}^{\chi_1})$ and $(\chi_2, \mathbb{C}^{\chi_2})$ are one-dimensional representations of $\exp \mathfrak{z}_{u_2}(\mathfrak{c_2})$ and $\exp\mathfrak{z}(\mathfrak{u})$, respectively.
We may assume that the conjugate representation $\overline{\pi}$ is realized in $\Gamma^{hol}(G/K,G\times_K\mathbb{C}^\chi)$, where $G\times_K\mathbb{C}^\chi$ is equipped with a structure of a $G$-equivariant holomorphic line bundle.
We have an extension $\chi:G_-\rightarrow \mathbb{C}^\times$ and an isomorphism $G\times_K\mathbb{C}^\chi\simeq GG_-\times_{G_-}\mathbb{C}^\chi$.
The group $G_-^*$ acts by the left regular representation on the space
\begin{equation*}
\widehat{\mathbb{C}^\chi}:=\left\{f:G_-^*\rightarrow \mathbb{C}\mid\begin{array}{c} f\text{ is holomorphic and }f(gh)=\chi^{-1}(h)f(g)\\\text{ for all }g\in G_-^*\text{ and }h\in G_-\end{array}\right\}.
\end{equation*}
The restriction map induces an injection
\begin{equation*}\widehat{\mathbb{C}^\chi}\subset \left\{f:U_\mathbb{C}\rightarrow\mathbb{C}\mid\begin{array}{c} f\text{ is holomorphic and }\\f(gh)=\chi^{-1}(h)f(g)\text{ for all }g\in U_\mathbb{C}\text{ and }h\in U_-\end{array}\right\},
\end{equation*}
and hence $\widehat{\mathbb{C}^\chi}$ is finite-dimensional.
Note that $GG_-^*=GG_-$ and the correspondence $\Gamma^{hol}(GG_-\times_{G_-}\mathbb{C}^\chi)\ni f\mapsto F\in\Gamma^{hol}(GG_-^*\times_{G_-^*}\widehat{\mathbb{C}^\chi})$ given by $F(g)(h):=f(gh)\,(g\in GG_-^*, h\in G_-^*)$ is an injective $G$-intertwining map.
Thus $\widehat{\mathbb{C}^\chi}$ is an irreducible $\exp \mathfrak{u}_2\times\exp\mathfrak{z}(\mathfrak{u})$-module and we may assume that $\overline{\pi}$ is realized in $\Gamma^{hol}(G\times_U\widehat{\mathbb{C}^\chi})$.
Let $\widetilde{\mathfrak{l}_1}:=\{\iota(T,0,0,0)\mid T\in\mathfrak{gl}(\nu,\mathbb{R})\}$, $\widetilde{\mathfrak{u}_2}:=\{\iota(0,K,0,0)\mid K\in\mathfrak{su}(r)\}$.
\begin{proposition}
If a finite-dimensional complex representation $\theta$ of $\widetilde{\mathfrak{g}_-^*}$ is irreducible when restricted to $\widetilde{\mathfrak{u}}$, then there exists some irreducible finite-dimensional representations $(\theta_1,V^{\theta_1})$ of $\mathfrak{gl}(\nu,\mathbb{C})$ and $(\theta_2, V^{\theta_2})$ of $\mathfrak{sl}(r,\mathbb{C})$, and $c_2\in\mathbb{C}$ such that
\begin{eqnarray*}
\theta(\iota(T_1,K_1,V,T_2+{}^tT_2)+i\iota(T_2,K_2,i\overline{V},-T_1-{}^t T_1))\\=\theta_1(T_1+iT_2)\otimes \mathrm{Id}_{V^{\theta_2}}+\mathrm{Id}_{V^{\theta_1}}\otimes \widetilde{\theta_2}(K_1+iK_2)\\+c_2\mathrm{tr}(K_1+iK_2)\mathrm{Id}_{V^{\theta_1}\otimes V^{\theta_2}}\in \mathfrak{gl}(V^{\theta_1}\otimes V^{\theta_2}),
\end{eqnarray*}
where $\widetilde{\theta_2}$ is the composition of the natural projection of $\mathfrak{gl}(r,\mathbb{C})$ onto $\mathfrak{sl}(r,\mathbb{C})$ with $\theta_2$.
\end{proposition}
\begin{proof}
We have a Levi decomposition of $\widetilde{\mathfrak{g}_-^*}$ if we define two complex subalgebras by
$V=0,\tr T_1=\tr T_2=\tr K_1=\tr K_2=0$ and by $T_1, T_2\in\mathbb{R}I_\nu, K_1, K_2\in\mathbb{R}I_r$.
By Schur's lemma, we see that $\theta(\iota(I_\nu,0,0,0))$ and $\theta(\iota(0,iI_r,0,0)$ are scalar multiples of the identity.
This implies that for $V\in\mathrm{Mat}_{\nu,r}(\mathbb{C})$, we have $\theta(\iota(0,0,V,0))=0$, which shows the assertion.
\qed
\end{proof}

Let $(\Theta_1,V^{\Theta_1})$ and $(\Theta_2,V^{\Theta_2})$ be holomorphic representations of the universal covering group of $GL(\nu,\mathbb{C})$ and $SL(r,\mathbb{C})$, respectively, and $p_1:G\rightarrow\exp(\widetilde{\mathfrak{l}_1})$, $p_{1,1}:G\rightarrow \exp(\mathfrak{z}(\widetilde{\mathfrak{l}_1}))\simeq\mathbb{R}$, $p_2:G\rightarrow\exp(\widetilde{\mathfrak{u}_2})\simeq SU(r)$, $p_{2,1}:G\rightarrow\exp(\mathfrak{z}(\widetilde{\mathfrak{u}_2}))\simeq \mathbb{R}$ the natural projections.
For linear forms $c_1\in(\mathfrak{z}(\widetilde{\mathfrak{l}_1})^*)_\mathbb{C}$, $c_2\in(\mathfrak{z}(\widetilde{\mathfrak{u}_2})^*)_\mathbb{C}$, let $(\exp x)^{c_1}:=e^{\langle c_1,x\rangle}\,(x\in\mathfrak{z}(\widetilde{\mathfrak{l}_1}))$, $(\exp x)^{c_2}:=e^{\langle c_2,x\rangle}\,(x\in\mathfrak{z}(\widetilde{\mathfrak{u}_2}))$, and for $g\in G$,
\begin{equation*}
m_{\Theta_1,c_2,\Theta_2}(g):=(p_{2,1}(g))^{c_2}\Theta_1(p_1(g))\otimes\Theta_2(p_2(g))\in GL(V^{\Theta_1}\otimes V^{\Theta_2}).
\end{equation*}
Define the action of $G$ on $\mathcal{D}\times V^{\Theta_1}\otimes V^{\Theta_2}$ by
\begin{equation*}
G\times\mathcal{D}\times V^{\Theta_1}\otimes V^{\Theta_2}\ni(g,z,v)\mapsto(p(g)z,m_{\Theta_1,c_2,\Theta_2}(g)v)\in\mathcal{D}\times V^{\Theta_1}\otimes V^{\Theta_2}.
\end{equation*}
Then we get a $G$-equivariant holomorphic vector bundle.
Conversely, every $G$-equivariant holomorphic vector bundle over $\mathcal{D}$ is isomorphic to one obtained in this way.
Indeed, by Tirao-Wolf \cite{TW}, for a $G$-invariant complex structure $j_0$ on $G/U$, and a finite-dimensional representation $(\tau, V^\tau)$ of $U$, the isomorphism classes of $G$-equivariant holomorphic vector bundles on $G\times_UV^\tau$ stand in one-one correspondence with extensions of the differential $d\tau$ of $\tau$ to complex representations of a certain complex subalgebra of $\mathfrak{g}_\mathbb{C}$ defined by the values of anti-holomorphic vector fields at $U\in G/U$.
Since $\mathfrak{u}_1\subset \mathfrak{k}$, we may assume that $\Theta_1$ is one-dimensional and hence $\Theta_1(p_1(g))=(p_{1,1}(g))^{c_1}\,(g\in G)$ for some $c_1\in(\mathfrak{z}(\widetilde{\mathfrak{l}_1})^*)_\mathbb{C}$.
\section{Equivalence classes of CS representations}
\label{sec:Equivalence}
We have proved the following theorem.
\begin{theorem}
For some $c_1\in(\mathfrak{z}(\widetilde{\mathfrak{l}_1})^*)_\mathbb{C}$, $c_2\in(\mathfrak{z}(\widetilde{\mathfrak{u}_2})^*)_\mathbb{C}$, and an irreducible unitary representation $(\Theta_2,V^{\Theta_2})$ of $SU(r)$, the representation $\pi$ or $\overline{\pi}$ is equivalent to a unitarization of the following representation:
\begin{eqnarray*}
\pi_{c_1,c_2,\Theta_2}(g)f(z):=(p_{1,1}(g))^{c_1}(p_{2,1}(g))^{c_2}\Theta_2(p_2(g))f(p(g^{-1})z)\\(g\in G,f\in\Gamma^{hol}(\mathcal{D},\mathcal{D}\times V^{\Theta_2}),z\in\mathcal{D}).
\end{eqnarray*}
\end{theorem}
For $c\in(\mathfrak{z}(\widetilde{\mathfrak{l}_1})^*)_\mathbb{C}$, let $\tau_c$ be the representation of $G$ on $\Gamma^{hol}(\mathcal{D},\mathcal{D}\times\mathbb{C})$ given by
\begin{equation*}
\tau_{c}(g)f(z)=(p_{1,1}(g))^{c}f(p(g^{-1})z)\quad(g\in G,f\in\Gamma^{hol}(\mathcal{D},\mathcal{D}\times\mathbb{C}) ,z\in\mathcal{D}).
\end{equation*}
The reproducing kernel of an invariant Hilbert space of $\Gamma^{hol}(\mathcal{D},\mathcal{D}\times V^{\Theta_2})$ must be of the form
\begin{eqnarray*}
K(Z,V,Z',V')=\det((Z-\overline{Z'})/2i-Q(V,V'))^\lambda\otimes \mathrm{Id}_{V^{\Theta_2}}\\((Z,V), (Z',V')\in\mathcal{D}, \lambda\in\mathbb{R})
\end{eqnarray*}
up to a constant, from which we see that $\pi_{c_1,c_2,\Theta_2}$ is unitarizable if and only if the representation $\tau_{c_1}$ of $G$ is unitarizable.
Since a maximal connected triangular Lie subgroup $B$ of $G$ acts simply transitively on $\mathcal{D}$, the unitarizability of $\tau_{c_1}$ is equivalent to the one of $\tau_{c_1}|_B$.
Let $c_1,c_1'\in(\mathfrak{z}(\widetilde{\mathfrak{l}_1})^*)_\mathbb{C}$, $c_2,c_2'\in(\mathfrak{z}(\widetilde{\mathfrak{u}_2})^*)_\mathbb{C}$, and $(\Theta_2,V^{\Theta_2})$, $(\Theta_2',V^{\Theta_2'})$ be irreducible unitary representations of $SU(r)$.
By the uniqueness (irreducibility) of a unitarization studied in \cite{Kobayashi}, we use the same symbol for a continuous representation and its unitarization (if it exists).
Then we have $\pi_{c_1,c_2,\Theta_2}\not\simeq\pi_{c_1',c_2',\Theta_2'}$ if $\tau_{c_1}\not\simeq \tau_{c_1'}$ as unitary representation of $G$.
Note that the latter condition is equivalent to $\tau_{c_1}|_B\simeq\tau_{c_1'}|_B$ (see \cite{Arashi20}).
Taking into account that $\pi_{c_1,c_2,\Theta_2}$ is the tensor product of two representations, we get the following theorem.
\begin{theorem}\label{thm:2}
If $\pi_{c_1,c_2,\Theta_2}$ and $\pi_{c_1'c_2',\Theta_2'}$ are unitarizable, then the following conditions are equivalent:
\begin{enumerate}
\item
$\pi_{c_1,c_2,\Theta_2}\simeq\pi_{c_1',c_2',\Theta_2'}$;
\item\label{cond2}
$\tau_{c_1}|_B\simeq\tau_{c_1'}|_B$, $c_2=c_2'$, and $\Theta_2\simeq \Theta_2'$.
\end{enumerate}
\end{theorem}

\begin{acknowledgement}
The author is supported by Tokyo Gakugei University and was supported by Foundation of Research Fellows, The Mathematical Society of Japan.
\end{acknowledgement}

\end{document}